\documentclass[11pt]{article}
\usepackage{enumerate}
\usepackage{amssymb,a4wide,latexsym,makeidx,epsfig,fleqn}
\usepackage{amsthm}
\usepackage{amsmath}
\usepackage{enumerate}
\usepackage{cite}   
\usepackage{tikz}
\newtheorem{theorem}{Theorem}[section]

\newtheorem{lemma}[theorem]{Lemma}

\newtheorem{problem}[theorem]{Problem}

\newtheorem{conjecture}[theorem]{Conjecture}

\begin{document}
	\textwidth 150mm \textheight 225mm
	\title{Size and spectral radius conditions for the existence of $k$-factor in graphs with given minimum degree
		\thanks{Supported by the National Natural Science Foundation of China (Nos. 12271439 and  12261073)}}
	\author{Zengzhao Xu$^{a,b}$, Bo Deng$^{c}$, Ligong Wang$^{a,b,}$\footnote{Corresponding author.}\\
		{\small $^{a}$ School of Mathematics and Statistics, Northwestern Polytechnical University,}\\
		{\small  Xi'an, Shaanxi 710129, P.R. China.}\\
		{\small $^b$ Xi'an-Budapest Joint Research Center for Combinatorics, Northwestern
			Polytechnical University,}\\
		{\small  Xi'an, Shaanxi 710129, P.R. China.}\\
		{\small $^{c}$ School of Mathematics and Statistics, Qinghai Normal University, Xining 810016, P.R. China.}\\
		{\small E-mail: xuzz0130@163.com; dengbo450@163.com; lgwangmath@163.com}\\}
	\date{}
	\maketitle
	\begin{center}
		\begin{minipage}{120mm}
			\vskip 0.3cm
\begin{center}
{\small {\bf Abstract}}
\end{center}
{\small In recent years, characterizing the conditions for the existence of an $[a,b]$-factor in graphs has received significant academic attention. For integers $b\ge a\ge 1$, an $[a, b]$-factor of a graph $G$ is defined as a spanning subgraph $H$ such that $a\le d_{H}(v)\le b$ for every $v\in V(G)$.  If $a=b=k$, a $[k, k]$-factor is called a $k$-factor of $G$. 
Let $\delta(G)$ denote the minimum degree of a graph $G$. Hao and Li [Electron. J. Combin. (2024)] posed the problem of characterizing the size or spectral radius conditions for a graph to contain an $[a,b]$-factor with $\delta(G)\ge a$. Fan, Lin and Zhu [Discrete
Math. (2026)] solved the size version of this problem for the case $a < b.$ In this paper, for $a=b=k$, we characterize the size condition for a graph to have a $k$-factor with $\delta(G)\ge k$. Using this result, we also characterize the spectral radius and distance spectral radius conditions for a graph to have a $k$-factor, respectively.

\vskip 0.1in \noindent {\bf Key Words}: \ $[a,b]$-factor, Size, Spectral radius, Distance spectral radius. \vskip
0.1in \noindent {\bf AMS Subject Classification }: \ 05C05. }
\end{minipage}
\end{center}

\section{Introduction}

In this paper, we only consider finite, undirected, and simple graphs. Let $G$ be a graph with vertex set $V(G)$ and edge set $E(G)$. Denote by $n=|V(G)|$ and $e(G)=|E(G)|$ the order and the size of $G$, respectively.
For a vertex $v \in V(G)$, we write $N_G(v)$ as its neighborhood, and its degree is defined as $d_G(v) = |N_G(v)|$. For a subgraph $G_0$ of $G$ and a vertex $v \in V(G_0)$, let $d_{G_0}(v)$ denote the degree of $v$ in $G_0$. The minimum degree of $G$ is denoted by $\delta(G)$. For disjoint vertex subsets $S_1, S_2  \subseteq V(G)$, define $E_G(S_1, S_2) = \{ xy \mid x \in S_1,\, y \in S_2,\, xy \in E(G) \}$, and let $e_G(S_1, S_2) = |E_G(S_1, S_2)|$. For two vertex-disjoint graphs $G_1$ and $G_2$, we use $G_1\cup G_2$
to denote the disjoint union of $G_1$ and $G_2$. The join of $G_1$ and $G_2$, denoted by $ G_1 \vee G_2$, is obtained from $G_1 \cup G_2$ by adding all possible edges between $ V(G_1)$ and $V(G_2)$. 
For a vertex subset $S \subseteq V(G)$, the subgraph induced by $S$ is denoted by $G[S]$ and let $G-S=G[V(G)\setminus S]$. For further details and related concepts, see \cite{BM}.

Let $G$ be a graph of order $n$ with vertex set $V(G) = \{ v_1, v_2, \dots, v_n \}$. The adjacency matrix of $G$ is the $n \times n$ $(0,1)$-matrix $A(G) = (a_{ij})$, where $a_{ij} = 1$ if $v_i$ and $v_j$ are adjacent in $G$ and $a_{ij} = 0$ otherwise. The largest eigenvalue of $A(G)$, denoted $\lambda(G)$, is called the spectral radius of $G$.

The distance matrix $D(G)$ of a connected graph $G$ is a matrix whose $(i,j)$-entry is the distance $d_G(v_i, v_j)$ between $v_i$ and $v_j$, and the distance is defined as the length of a shortest path from $v_i$ to $v_j$.  The largest
eigenvalue of the distance matrix $D(G)$ is called the distance spectral radius of
$G$, denoted by $\mu(G)$.

A factor is a spanning subgraph subject to specified conditions. The
study of factors in graphs has a long and rich history. Since Danish mathematician Petersen first attempted the study of factors in 1891, the graph factors theory play an important role in graph theory \cite{AK,C2,T2}. Studying the existence of graph factors helps reveal the structure of graphs and provides methods for solving practical problems, such as those in network design and combinatorial optimization. 
For integers $b\ge a\ge 1$, an $[a, b]$-factor of a graph $G$ is defined as a spanning subgraph $H$ such that $a\le d_{H}(v)\le b$ for every $v\in V(G)$.  If $a=b=k$, a $[k, k]$-factor is called a $k$-factor of $G$. In addition, a $1$-factor is also called a perfect matching. 

 Much effort has been devoted to finding sufficient conditions for the existence of $[a, b]$-factors in a graph by using various graphic parameters such as eigenvalues \cite{BH,C,CG,CGH,FL,FLL,FLLO,HL2,JLL,LXO}, 
 stability number \cite{KL}, size \cite{HL,HV,J}, toughness \cite{CFL,EJKS,LCLK}, independence number \cite{N2}. In 2021, Cho, Hyum, O and Park \cite{CHOP} posed the spectral version conjecture for the existence of $[a, b]$-factors in graphs.

\begin{conjecture}(\cite{CHOP})\label{con1}
	Let $G$ be an $n$-vertex graph and $b\ge a$ be two positive integers, where $na\equiv 0 \pmod{2}$ and $n \geq a + 1$. If 
	$$\lambda(G) > \lambda(K_{a-1} \lor (K_1 \cup K_{n-a})),$$
	then $G$ contains an $[a,b]$-factor.
\end{conjecture}
Over the past few years, researchers have studied this conjecture and obtained some results. Fan, Lin and Lu \cite{FLL} confirmed Conjecture \ref{con1} for $n\ge 3a+b-1.$ The full conjecture was confirmed by Wei and Zhang \cite{WZ}. They \cite{WZ} initially established the size condition that guarantee the existence of an $[a,b]$-factor in a graph, which they subsequently employed to prove a spectral radius condition.
\begin{theorem}\label{wz}(\cite{WZ})
	Let $a, b$ be two positive integers with $b \ge a$, and let $G$ be a graph of order $n$ with minimum degree $\delta(G) \ge a$, where $n \ge a+1$. If
		$$e(G) \ge \binom{n-1}{2} + \frac{a+1}{2}$$
	and $na \equiv 0 \pmod{2}$ when $a = b$, then $G$ has an $[a, b]$-factor.
\end{theorem}
\begin{theorem}\label{wz2}(\cite{WZ})
	Let $a, b$ be two positive integers with $a \le b$, and let $G$ be a graph of order $n$, where $n \ge a+1$. If $$\lambda(G) > \lambda(K_{a-1} \vee (K_{n-a} \cup K_1))$$ and $na \equiv 0 \pmod 2$ when $a = b$, then $G$ has an $[a, b]$-factor.
\end{theorem}
 Hao and Li \cite{HL} determined the largest spectral radius among all $n$-vertex graphs forbidding $[a,b]$-factors and characterized the extremal graphs, which strengthened the result of Wei and Zhang \cite{WZ}. In addition, Hao and Li \cite{HL} observed that $\delta(G)\ge a$ is a necessary condition for the existence of an $[a,b]$-factor in a graph, and they posed the following problem.
\begin{problem} (\cite{HL})\label{prob1}
	Determine sharp lower bounds on the size or spectral radius of an $n$-vertex graph $G$ with $\delta(G) \geq a$ such that $G$ contains an $[a, b]$-factor.
\end{problem}
Recently, Problem \ref{prob1} has attracted widespread attention from scholars, and the spectral version of Problem \ref{prob1} has been solved by Tang, Zhang \cite{TZ} and Fan, Lin, Zhu \cite{FLZ}  for the cases $a = b$ and $a < b$, respectively. Fan, Lin and Zhu \cite{FLZ} also solved the size version of Problem \ref{prob1} for the case $a < b.$
\begin{theorem}\label{T-FLZ}(\cite{FLZ})
	Let $a$ and $b$ be two positive integers with $b > a$, and let $G$ be a connected graph of order $n \ge 4a + \frac{5b}{2} + 6$ with minimum degree $\delta(G) \ge a$. If
	\[
	e(G) \ge \binom{n - b - 1}{2} + ab + 2a,
	\]
	then $G$ contains an $[a, b]$-factor.
\end{theorem}

In this paper,  we solve the size version of Problem \ref{prob1} for the case $a = b=k.$

\begin{theorem}\label{T2}
	Let $k, n$ be positive integers such that $k\ge1$, $kn\equiv 0 \pmod 2$ and $n\ge 12k+40$. Let $G$ be a graph of order $n$ with $\delta(G)\ge k$.  If
	$$e(G)\ge \binom{n-k-1}{2}+k^2+2k, $$ then $G$ has a $k$-factor.
\end{theorem}

Combining Theorems \ref{T-FLZ} and \ref{T2}, we can obtain the following result, which completely resolves  the size condition of Problem \ref{prob1}.
\begin{theorem}\label{T-2.1}
	Let $a$ and $b$ be two positive integers with $b\ge a$, and let $G$ be a connected graph of order $n \ge \frac{19a}{2} + \frac{5b}{2} +40 $ with $\delta(G) \ge a$. If
	\[
	e(G) \ge \binom{n - b - 1}{2} + ab + 2a,
	\]
and $an\equiv 0 \pmod 2$ when $a=b$,	then $G$ contains an $[a, b]$-factor.
\end{theorem}
{\bf Remark 1:} 
Compared with Theorem \ref{wz}, if $a\le b$ and $n\ge a+\frac{b}{2}+3$, then
\begin{align*}
	&\quad\binom{n-1}{2}+\frac{a+1}{2}-\bigl(\binom{n - b - 1}{2} + ab + 2a\bigr)=\frac{b}{2}(2n-(b+3)-\frac{a}{b}(2b+3)+\frac{1}{b})>0.
\end{align*}
Hence Theorem \ref{T-2.1} improves Theorem \ref{wz}. In fact, the result of Theorem \ref{T-2.1} achieves the best possible condition. Suppose that $H_n^{a,b}$ is a graph obtained from $K_a \vee (K_{n-a-b-1} \cup (b+1)K_1)$ by adding $a-1$ edges between one vertex in $(b+1)K_1$ and $a-1$ vertices in $K_{n-a-b-1}$. Note that $$e(H_n^{a,b})=\binom{n - b - 1}{2} + ab + 2a-1. $$
We can show that $H_n^{a,b}$ contains no $[a,b]$-factors. In fact, if $H_n^{a,b}$ contains a $[a,b]$-factor $F$, then there are at least $(b+1)a-(a-1)=ab+1$ edges in $F$ connecting   $V((b+1)K_1)$ to $V(K_a)$. It follows that $d_F(w)>b$ for some vertex $w$ in $V(K_a)$, which contradicts $ a\le d_F(w)\le b$. Hence the
condition in Theorem \ref{T-2.1} is best possible.

Using Theorem \ref{T-2.1}, we also characterize the spectral radius and distance spectral radius conditions for a graph to have a $k$-factor, respectively.

\begin{theorem}\label{T3}
	Let $a, b$ be two positive integers with $b\ge a$. Let $G$ be a connected graph of order $n \ge \frac{19a}{2} + \frac{5b}{2}+\frac{ab}{2} +40 $ with  $\delta(G) \ge a$. If
	\[
	\lambda(G) \ge n - b - 1,
	\]
and $an\equiv 0 \pmod 2$ when $a=b$,	then $G$ contains an $[a, b]$-factor.
\end{theorem}
{\bf Remark 2:} Note that $K_{n-1}$ is a spanning subgraph of $K_{a-1} \vee (K_{n-a} \cup K_1)$. Compared with Theorem \ref{wz2}, we have $$\lambda(K_{a-1} \vee (K_{n-a} \cup K_1))>\lambda(K_{n-1})=n-2\ge n-b-1, \quad b\ge a\ge 1.$$
Hence, we slightly improve the lower bound of Theorems \ref{wz2} by restricting the condition of $\delta(G)\ge a$.
\begin{theorem}\label{T4}
	Let $a, b$ be two positive integers with $b\ge a$. Let $G$ be a connected graph of order $n \ge \frac{19a}{2} + \frac{5b}{2} +40 $ with  $\delta(G) \ge a$. If
	\[
	\mu(G) \le n+2b+1-\frac{(b+2a+1)(b+2)}{n},
	\]
and $an\equiv 0 \pmod 2$ when $a=b$,	then $G$ contains an $[a, b]$-factor.
\end{theorem}

\section{Proof of Theorem \ref{T2}}

In this section, we prove Theorem \ref{T2}, which characterizes the size condition for a graph to contain an $k$-factor. For two disjoint subsets $S,T\subseteq V(G)$, let $q(S, T)$ denote the number of components $C$ in  $G - (S \cup T)$ such that  $e_G(V(C), T) + k|V(C)| \equiv 1 \pmod 2 $, and let
\[
\theta(S, T) = k|S|- k|T| + \sum_{v \in T} d_{G-S}(v)  - q(S, T).
\]

In 1952, Tutte \cite{T} established the following well-known necessary and sufficient condition for a graph $G$ to have a $k$-factor. This fundamental theorem laid the foundation for subsequent research on $k$-factors.
\begin{theorem}(\cite{T} )\label{T2.1}
	Let $G$ be a graph and $k$ be a positive integer. Then $G$ has a $k$-factor if and only if $$\theta(S, T) = k|S|- k|T| + \sum_{v \in T} d_{G-S}(v)  - q(S, T) \ge 0$$ for all disjoint subsets $S, T \subseteq V(G)$. Moreover, $\theta(S, T) \equiv k|V(G)| \pmod 2$ for all disjoint subsets $S, T \subseteq V(G)$.
\end{theorem}

Now, we prove Theorem \ref{T2}.

\begin{proof}[\bf Proof of Theorem \ref{T2}]  %
		Let $B_k(n)= \binom{n-k-1}{2}+k^2+2k-1$. By the assumption of Theorem \ref{T2}, we have
	\begin{equation}\label{bk}
		e(G)\ge B_k(n)+1.
	\end{equation}
Our proof proceeds by contradiction. We assume that $G$ contains no $k$-factor, which yields $e(G)\le B_k(n).$ This contradicts assumption of Theorem \ref{T2} that $e(G)\ge B_k(n)+1$, thus completing the proof.
	
For disjoint subsets $S, T \subseteq V(G)$, let $q(S, T)$ be the number of $\{C: C$ is a component of $G - (S \cup T)$ such that $k|V(C)|+e_G(V(C), T) \equiv 1 \pmod{2}\}$.
Let $$\theta(S,T)=k|S| - k|T| + \sum_{v \in T} d_{G-S}(v) - q(S, T).$$ Then we have the following claim.

{\bf Claim 1.} $\theta(S, T) \equiv 0 \pmod{2}$ for any disjoint subsets $S, T \subseteq V(G)$. 
\begin{proof} Since $kn \equiv 0 \pmod{2}$ and $\theta(S,T) \equiv kn \pmod{2}$ due to Theorem \ref{T2.1}, we have $\theta(S, T) \equiv 0 \pmod{2}$ for any disjoint subsets $S, T \subseteq V(G)$. 
\end{proof}
Suppose that $G$ is a connected graph with $\delta(G)\ge k$ that has no $k$-factor, where $k\ge 1$. By Theorem \ref{T2.1} and Claim 1, there exist disjoint subsets $S,T \subseteq V(G)$ such that $\theta(S,T)=k|S| - k|T| + \sum_{v \in T} d_{G-S}(v) - q(S, T)\le -2.$  We can choose $S$ and $T$ such that $|S\cup T|$ is maximum. Let $s=|S|$, $t=|T|$ and $q=q(S,T)$. Then
\begin{equation}\label{eq:1}
	\theta(S,T)\le -2 \iff \sum_{v \in T} d_{G-S}(v) \leq kt - ks + q-2.
\end{equation}
 
Since $\delta(G)\ge k$, we have $\sum_{v \in T}d_{G-S}(v)\ge (k-s)t$. By (\ref{eq:1}), we have
\begin{equation}\label{eq:3}
	q\ge s(k-t)+2.
\end{equation}

Note that $\sum_{v \in T}d_{G-S}(v)\ge 0$, by (\ref{eq:1}), we have $ ks\le kt+q-2$, which implies
\begin{equation}\label{eq:4}
	s\le t+\frac{q-2}{k}.
\end{equation}

Let $U=G-(S\cup T)$ and $m=\max\{3, k+2-t\}$.

{\bf Claim 2.} For any $v\in U$, we have
		$d_{G-S}(v)\ge k+1$ and  $e_G(v,T)\le k-1.$ Hence $|C|\ge m=\max\{3, k+2-t\}$ for any component $C\in G[U].$
\begin{proof}
	Otherwise, assume that $d_{G-S}(x)\le k$ for some $x\in U$. Let $T_1 = T \cup \{x\}$ and $q'=q(S,T_1)$. Then $|S \cup T_1|=|S \cup T|+1$ and $q'\ge q-1$. By (\ref{eq:1}), we obatin
	\[
	\begin{aligned}
		\theta(S, T_1)&=k|S| - k|T_1| + \sum_{v \in T_1} d_{G-S}(v) - q' \\
		& =ks-k(t+1)+\sum_{v \in T} d_{G-S}(v)+d_{G-S}(x)-q'\\
		&\le ks - kt-k+\sum_{v \in T} d_{G-S}(v)+k-(q-1) \quad \text{(since $q'\ge q-1$)}\\
		&= \theta(S, T)+1\le -1\quad (\text{by } (\ref{eq:1})).
	\end{aligned}
	\]
By Claim 1, we have $\theta(S, T_1)\le -2$, which contradicts  the maximality of $|S \cup T|$.	Hence $d_{G-S}(v)\ge k+1$ for any $v\in U$.
	
	Suppose that $e_G(y,T)\ge k$ for some $y\in U$.  Let $S_1 = S \cup \{y\}$ and $q''=q(S_1,T)$. Then $|S_1 \cup T|=|S \cup T|+1$ and $q''\ge q-1$. By (\ref{eq:1}), we obtain
	\[
	\begin{aligned}
		\theta(S_1, T)&=k|S_1| - k|T| + \sum_{v \in T} d_{G-S_1}(v) - q'' \\
		& =k(s+1)-kt+\sum_{v \in T} d_{G}(v)-e_G(S_1, T)-q''\\
		&=k(s+1)-kt+\sum_{v \in T} d_{G}(v)-(e_G(S, T)+e_G(y, T))-q''\\
		&\le ks+k-kt+\sum_{v \in T} d_{G-S}(v)-k-(q-1) \quad \text{(since $e_G(y,T)\ge k$ and $q'\ge q-1$)}  \\
		&\le \theta(S, T)+1\le -1 \quad (\text{by } (\ref{eq:1})).
	\end{aligned}
	\]
	By Claim 1, we have $\theta(S_1, T)\le -2$, which contradicts  the maximality of $|S \cup T|$. Hence $e_G(v,T)\le k-1$ for any $v\in U$.
	
	Let $C$ be a component of $G[U]$. Note that $e_G(v,T)\le t$ and $	d_{G-S}(v)=d_C(v)+e_G(v,T)$ for $v\in C$. Then 
	\[
	d_C(v)=d_{G-S}(v)-e_G(v,T)\ge k+1-(k-1)=2,
	\]
	\[
	d_C(v)=d_{G-S}(v)-e_G(v,T)\ge k+1-t.
	\]
	Since $|V(C)|\ge d_C(v)+1$, we obtain $|C|\ge m=\max\{3, k+2-t\}$.
\end{proof}

Let $u=|U|=|V(G)\setminus(S\cup T)|$ and let \(C_1,C_2,\ldots,C_l\) be the components of \(G[U]\), where \(l\ge q\). If $q\ge 1$, without loss of generality, suppose that \(C_1,\ldots,C_q\) are
the components such that $k|V(C_i)|+e_G(V(C_i),T)\equiv 1\pmod 2$. Write
$c_i=|V(C_i)|$ for  $1\le i\le l.$ Partition \(C_1,\ldots,C_l\) into \(q\) nonempty groups such that for $1\le i\le q$,
the \(i\)-th group contains \(C_i\), and let \(x_i\) be the
number of vertices in the \(i\)-th group (see Figure \ref{figure:1}).

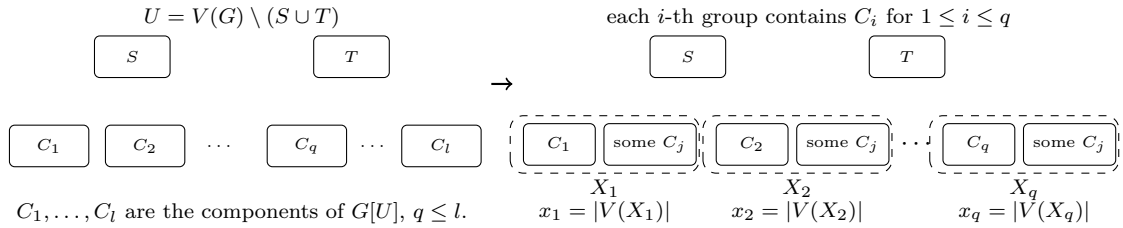
\begin{figure}[htbp]
	\centering
\begin{tikzpicture}[font=\tiny, scale=1]

	\node[font=\scriptsize] at (0,1.38)
	{$U=V(G)\setminus(S\cup T)$};
	
	\node[draw, rounded corners=2pt, minimum width=1.0cm,
	minimum height=0.55cm] at (-1.45,0.86) {$S$};
	
	\node[draw, rounded corners=2pt, minimum width=1.0cm,
	minimum height=0.55cm] at (1.45,0.86) {$T$};
	
	\node[draw, rounded corners=2pt, minimum width=1.05cm,
	minimum height=0.55cm] at (-2.55,-0.32) {$C_1$};
	
	\node[draw, rounded corners=2pt, minimum width=1.05cm,
	minimum height=0.55cm] at (-1.28,-0.32) {$C_2$};
	
	\node at (-0.28,-0.32) {$\cdots$};
	
	\node[draw, rounded corners=2pt, minimum width=1.05cm,
	minimum height=0.55cm] at (0.86,-0.32) {$C_q$};
	
	\node at (1.75,-0.32) {$\cdots$};
	
	\node[draw, rounded corners=2pt, minimum width=1.05cm,
	minimum height=0.55cm] at (2.64,-0.32) {$C_l$};
	
	\node[font=\scriptsize, align=center] at (0,-1.20)
	{$C_1,\ldots,C_l$ are the components of $G[U]$, $ q\le l$.};
	
	\draw[->, thick] (3.30,0.5) -- (3.58,0.5)
	node[midway, above=1pt, font=\scriptsize] {};

	\node[font=\scriptsize, align=center] at (7.55,1.38)
	{each $i$-th group contains $C_i$ for $1\le i\le q$ };
	
	\node[draw, rounded corners=2pt, minimum width=1.0cm,
	minimum height=0.55cm] at (5.90,0.86) {$S$};
	
	\node[draw, rounded corners=2pt, minimum width=1.0cm,
	minimum height=0.55cm] at (8.80,0.86) {$T$};
	
	\node[draw, rounded corners=2pt, minimum width=0.95cm,
	minimum height=0.55cm] at (4.20,-0.30) {$C_1$};
	
	\node[draw, rounded corners=2pt, minimum width=1.15cm,
	minimum height=0.55cm] at (5.40,-0.30)
	{$\mathrm{some}\ C_j$};
	
	\draw[dashed, rounded corners=5pt]
	(3.55,-0.68) rectangle (6.05,0.08);
	
	\node[font=\scriptsize, align=center] at (4.80,-1.05)
	{$ X_1$\\[-1pt]$x_1=|V( X_1)|$};
	
	\node[draw, rounded corners=2pt, minimum width=0.95cm,
	minimum height=0.55cm] at (6.75,-0.30) {$C_2$};
	
	\node[draw, rounded corners=2pt, minimum width=1.15cm,
	minimum height=0.55cm] at (7.95,-0.30)
	{$\mathrm{some}\ C_j$};
	
	\draw[dashed, rounded corners=5pt]
	(6.10,-0.68) rectangle (8.60,0.08);
	
	\node[font=\scriptsize, align=center] at (7.35,-1.05)
	{$ X_2$\\[-1pt]$x_2=|V(X_2)|$};
	
	\node[font=\small] at (8.95,-0.30) {$\cdots$};
	
	\node[draw, rounded corners=2pt, minimum width=0.95cm,
	minimum height=0.55cm] at (9.75,-0.30) {$C_q$};
	
	\node[draw, rounded corners=2pt, minimum width=1.15cm,
	minimum height=0.55cm] at (10.95,-0.30)
	{$\mathrm{some}\ C_j$};
	
	\draw[dashed, rounded corners=5pt]
	(9.10,-0.68) rectangle (11.60,0.08);
	
	\node[font=\scriptsize, align=center] at (10.35,-1.05)
	{$ X_q$\\[-1pt]$x_q=|V( X_q)|$};
	
\end{tikzpicture}
\caption{Partition into $q$ groups } 
\label{figure:1} 
\end{figure}

 By Claim 2, we obtain 
\[
x_i\ge c_i\ge m,
\qquad
\sum_{i=1}^{q}x_i=u,\qquad 	u\ge mq.
\] 

Let \(L\) denote the number of missing edges between   \(C_1,C_2,\ldots,C_l\). 

{\bf Claim 3.} $L\ge m(q-1)(u-\frac{mq}{2})$ and $q-2-L\le -1$.
\begin{proof}
Note that $L\ge 0.$ If $q=0$, then we obtain $L\ge 0>  m(0-1)(u-\frac{mq}{2})$ and  $q-2-L\le -2< -1.$ Hence  $q-2-L\le -1$. Next, we consider $q\ge 1$.
 
Every pair of vertices
from different groups is counted by \(L\). Hence
\[
L\ge \sum_{1\le i<j\le q}x_ix_j.
\]
Let \(x_i=m+y_i\), then we have \(y_i\ge0\) and $\sum_{i=1}^{q}y_i=u-mq.$
Therefore,
\begin{align*}
	L
	&\ge \sum_{1\le i<j\le q}(m+y_i)(m+y_j)\\
	&=\binom{q}{2}m^2
	+m(q-1)\sum_{i=1}^{q}y_i
	+\sum_{1\le i<j\le q}y_i y_j\\
	&\ge \binom{q}{2}m^2+m(q-1)(u-mq)\\
	&=m(q-1)\left(u-\frac{mq}{2}\right).
\end{align*}
For $q\ge1$, we have $L\ge m(q-1)(u-\frac{mq}{2})\ge q-1$. 
It follows that	$q-2-L\le-1.$
\end{proof}

Note that $e_G(S\cup U)\le \binom{n-t}{2}-L $ and $e_G(S,T)\le st$. By (\ref{eq:1}), we obtain
\begin{equation}\label{eq:2}
	\begin{aligned}
		e(G)
		&=e_G(S\cup U)+e_G(S,T)+e_G(T)+e_G(T,U)\\
		&\le\binom{n-t}{2}-L+st+\sum_{v \in T} d_{G-S}(v)\\
		&\le\binom{n-t}{2}+kt+s(t-k)+q-2-L.
	\end{aligned}
\end{equation}

The remaining proof of Theorem \ref{T2} is partitioned based on the value of $t$.

{\bf Case 1: $0\le t\le k$.}

Since $t\le k$, by (\ref{eq:3}), we have $q\ge s(k-t)+2\ge2$.

{\bf Subcase 1.1} $t=k$.

In this case, we obtain $m=\max\{3, k+2-t\}$=3 and $
u\ge mq= 3q\ge 6.$ Hence $u-\frac{3q}{2}\ge \frac{u}{2}$ and
\begin{equation}\label{eq:5}
	L\ge m(q-1)(u-\frac{mq}{2}) =3(q-1)(u-\frac{3q}{2})\ge \frac{3u}{2}(q-1).
\end{equation}

Substitute $t=k$ into (\ref{eq:2}) and using (\ref{eq:5}), we obtain
\begin{equation*}
	\begin{aligned}
		e(G)&\le\binom{n-t}{2}+kt+s(t-k)+q-2-L\\
		&=\binom{n-k-1}{2}+k(k+1)+k-1+n-3k+q-2-L\\
		&\le \binom{n-k-1}{2}+k(k+1)+k-1+n-3k+q-2-\frac{3u}{2}(q-1)
	\end{aligned}
\end{equation*}

Since $n=s+k+u$ and (\ref{eq:4}), we have $n-3k=u+s-2k
\le u-k+\frac{q-2}{k}.$ Note that $B_k(n)=\binom{n-k-1}{2}+k(k+1)+k-1$.

Since $k\ge 1$, $u\ge 6$ and (\ref{eq:5}), we have
\begin{align*}
	e(G)-B_k(n)
	&\le n-3k+q-2
	-3(q-1)\left(u-\frac{3q}{2}\right)\\
	&\le u-k+\left(1+\frac{1}{k}\right)(q-2)
	-\frac{3u}{2}(q-1)\\
	&=-k-\frac{u}{2}
	+(q-2)\left(1+\frac{1}{k}-\frac{3u}{2}\right)<0,
\end{align*}
which contradicts (\ref{bk}). 
 
{\bf Subcase 1.2} $t<k$.
 
 In this case, $m=\max\{3, k+2-t\}=k+2-t$. By Claim 3, we have
 \begin{equation}\label{eq:6}
 	\begin{aligned}
 		L&\ge m(q-1)\left(u-\frac{mq}{2}\right)\\
 		&=m(q-2)(u-\frac{mq}{2})+m(u-\frac{mq}{2})\\
 		&=m(q-2)(u-\frac{mq}{2})+m(u-\frac{m(q-2)+2m}{2})\\
 		&=m(q-2)(u-\frac{m(q+1)}{2})+m(u-m).
 	\end{aligned}
 \end{equation}

Note that $n=s+t+u$ and $t-k=2-m$. By  (\ref{eq:4}), (\ref{eq:2}) and (\ref{eq:6}), we have
 \begin{equation*}
 	\begin{aligned}
 		e(G)&\le \binom{n-t}{2}+kt+s(t-k)+q-2-L\\
 		&\le\binom{n-t}{2}+kt+s(t-k)+q-2
 		-m(q-2)(u-\frac{m(q+1)}{2})-m(u-m)\\
 		&=\binom{n-t}{2}+kt+s(2-m)-m(n-s-t-m)
 		\\
 		&\hspace{2em}
 		+(q-2)\left[1-m\left(u-\frac{m(q+1)}2\right)\right]\\
 		&=\binom{n-t}{2}+kt
 		-m(n-t-m)+2s\\
 		&\hspace{2em}
 		+(q-2)\left[1-m\left(u-\frac{m(q+1)}2\right)\right]\\
 		&\le\binom{n-t}{2}+(k+2)t-m(n-t-m)\\
 		&\hspace{2em}
 		+(q-2)\left[1+\frac2k
 		-m\left(u-\frac{m(q+1)}2\right)\right] \quad \text{(since $s\le t+\frac{q-2}{k}$)}  \\
 	\end{aligned}
 \end{equation*}
Let $g(q)=(q-2)\left[1+\frac2k
-m\left(u-\frac{m(q+1)}2\right)\right]$. Now, we prove that $g(q)\le 0.$
	
Note that $q\ge s(k-t)+2\ge2$. If $q=2$, then $g(2)=0$.
	If $q\ge3$, then $u\ge mq$ and $m=\max\{3, k+2-t\}\ge 3$ imply
$	m\left(u-\frac{m(q+1)}2\right)
	\ge \frac{m^2(q-1)}2
	\ge m^2>9>1+\frac2k $. Hence $g(q)\le0$, which implies
	$$e(G) \le\binom{n-t}{2}+(k+2)t-m(n-t-m).$$

	Since $m=k+2-t$, we have $n-t-m=n-k-2$. Recall that $t<k$, we have
	\begin{align*}
		e(G)
		&\le\binom{n-t}{2}+(k+2)t-(k+2-t)(n-k-2)\\
		&= \frac{1}{2}n^2-(k+\frac{5}{2})n + \frac{1}{2} ( t^2+t+2(k+2)^2)\\
		&=B_k(n)+\frac{ k - k^2 + t + t^2+8}{2}-n\\
		&\le B_k(n)+ k+4-n<B_k(n),
	\end{align*}
which contradicts (\ref{bk}).

{\bf Case 2:} $k+1\le t\le \frac{n}{4}$.
	
	In this case, we obtain $m=\max\{3, k+2-t\}$=3. We prove the following claim.
	
{\bf Claim 4.}
		$e(G)\le\binom{n-t}{2}+t(t-1)+k-1.$
\begin{proof} Since $0\le \sum_{v \in T} d_{G-S}(v) \leq kt - ks + q-2,$ we obtin \begin{equation}\label{eq:7}
		q\ge k(s-t)+2.
	\end{equation}
	If $q\le1$, then $k(s-t)\le-1$ due to (\ref{eq:7}), which implies $s\le t-1$. Since $t-k>0$, by Claim 3 and (\ref{eq:2}), we obtain
	\begin{align*}
		e(G)&\le\binom{n-t}{2}+kt+s(t-k)+q-2-L\\
		&\le\binom{n-t}{2}+kt+s(t-k)-1\\
		&\le\binom{n-t}{2}+kt+(t-1)(t-k)-1\quad \text{(since $s\le t-1$)}\\
		&=\binom{n-t}{2}+t(t-1)+k-1.
	\end{align*}
	
	Now suppose that $q\ge2$.
	Since $u\ge mq=3q$, by (\ref{eq:4}), we have
	\[
	q\le\frac u3,\qquad
	s\le t+\frac{q-2}{k}\le t+q\le t+\frac u3.
	\]
	Hence $n=s+t+u\le2t+\frac{4u}{3}$. Since $t\le\frac{n}{4}$, then
	\[
	u\ge\frac{3}{4}(n-2t)\ge\frac{3}{4}(4t-2t)=\frac{3}{2}t .
	\]
	Since $q\ge 2$, $m=3$ and $u\ge mq=3q$, we have 
	$$\frac{9}{4}(q-1)\ge q,\quad u-\frac{3q}{2}\ge \frac{u}{2}. $$
Sine $m=3$, $ u-\frac{3q}{2}\ge \frac{u}{2}$, $u\ge \frac{3}{2}t$, by Claim 3, we obtain
	\[
L\ge3(q-1)\left(u-\frac{3q}{2}\right)
\ge 3(q-1)\frac{u}{2}
\ge\frac{9}{4}(q-1)t\ge tq.
\]

Then
	\begin{align*}
		e(G)&\le\binom{n-t}{2}+kt+s(t-k)+q-2-L\\
		&\le\binom{n-t}{2}+kt+
		\left(t+\frac{q-2}{k}\right)(t-k)+q-2-L\\
		&=\binom{n-t}{2}+t^2+\frac{t}{k}(q-2)-L\\
		&\le\binom{n-t}{2}+t^2+t(q-2)-tq\\
		&=\binom{n-t}{2}+t^2-2t\\
		&\le\binom{n-t}{2}+t(t-1)+k-1.
	\end{align*}
	It follows that Claim 4 holds.​ 
\end{proof}
By Claim 4, we obtain
\begin{align*}
	e(G)-B_k(n)
&\le\binom{n-t}{2}+t(t-1)+k-1-B_k(n)\\
&=\frac{(t-k-1)(3t+3k+2-2n)}2<0,
\end{align*}
which contradicts (\ref{bk}). 

	{\bf Case 3:} $t>\frac{n}{4}$.
	
	Since $\frac{n}{4}\ge\frac{12k+40}{4}>k$, we have $t-k>0$.
	Since $s\le n-t$, by Claim 3, we obtain
	\begin{align*}
		e(G)&\le\binom{n-t}{2}+kt+s(t-k)+q-2-L\\
		&\le\binom{n-t}{2}+kt+s(t-k)-1\\
		&\le\binom{n-t}{2}+kt+(n-t)(t-k)-1\\
		&=\frac{n^2}{2}-(k+\frac{1}{2})n + \frac{1}{2} (  - t^2+ 4 k t+ t-2)\\
		&=B_k(n)+n + \frac{1}{2}(-t^2+ 4 k t+ t- k (3 k+7)-2)\\
		&\le B_k(n)+n +\frac{1}{2}(-\frac{n^2}{16}+(k+\frac{1}{4})n- k (3 k+7)-2)\\
		&=B_k(n)+(-\frac{n^2}{32}+\frac{4k+9}{8}n-\frac{1}{2} k (3 k+7)-1)\le B_k(n),
	\end{align*}
which contradicts (\ref{bk}). 

This completes the proof of Theorem \ref{T2}.
\end{proof}

\section{Proofs of Theorems \ref{T3} and \ref{T4}}

In this section, using Theorem \ref{T-2.1}, we can characterize the spectral radius and distance spectral radius conditions for a graph to have a $k$-factor, respectively.

\begin{lemma}(\cite{HSF,N})\label{le:3.1} \ Let $G$ be a graph on $n$ vertices and $m$ edges with minimum degree $\delta(G) \geq 1$. Then
	$
	\lambda(G) \leq \frac{\delta(G) - 1}{2} + \sqrt{2e(G) - n\delta(G) + \frac{(\delta(G) + 1)^2}{4}},
	$
	with equality if and only if $G$ is either a $\delta(G)$-regular graph or a bidegreed graph in which each vertex is of degree either $\delta$ or $n - 1$.
\end{lemma}

\begin{lemma}(\cite{HSF,N})\label{le:3.2} \ For nonnegative integers $p$ and $q$ with $2q \leq p(p - 1)$ and $0 \leq x \leq p - 1$, the function
	$
	f(x) = \frac{x - 1}{2} + \sqrt{2q - px + \frac{(1 + x)^2}{4}}
	$ 
	is decreasing with respect to $x$.
\end{lemma}

By​ Lemmas \ref{le:3.1}, \ref{le:3.2} and Theorem  \ref{T-2.1}, we prove Theorem \ref{T3}. 

\begin{proof}[\bf Proof of Theorem \ref{T3}]  %
	Let $a, b, n$ be two positive integers and $an\equiv 0 \pmod 2$ when $a=b$. Let $G$ be a connected graph of order $n \ge \frac{19a}{2} + \frac{5b}{2}+\frac{ab}{2} +40 $ with  $\delta(G) \ge a$. By the assumption, we have $$\lambda(G) \ge n - b - 1.$$
 By Lemma \ref{le:3.1}, we have
	\[
	n - b - 1 < \lambda(G) \le \frac{\delta(G) - 1}{2} + \sqrt{2e(G) - \delta n + \frac{(\delta(G) + 1)^2}{4}}.
	\]
	According to Lemma \ref{le:3.2} , we obtain
	\[
	n - b - 1 < \frac{a  - 1}{2} + \sqrt{2e(G) - an + \frac{(a  + 1)^2}{4}}.
	\]
	It follows that
	\[
	e(G) > \frac{1}{2}n^2 - \left(b + \frac{1}{2}\right)n + \frac{b(a +b+ 1)}{2}>\binom{n - b - 1}{2} + ab + 2a.
	\]
	By Theorem  \ref{T-2.1}, $G$ contains an $[a, b]$-factor.
\end{proof}

Now, using Theorem  \ref{T-2.1}, we prove Theorem \ref{T4}.

\begin{proof}[\bf Proof of Theorem \ref{T4}]  %
Let $a, b, n$ be two positive integers and $an\equiv 0 \pmod 2$ when $a=b$.	Let $G$ be a connected graph of order $n \ge \frac{19a}{2} + \frac{5b}{2} +40 $ with  $\delta(G) \ge a$. By the assumption, we have $$\mu(G) \le n+2b+1-\frac{(b+2a+1)(b+2)}{n}.$$
	Let $W(G) = \sum_{i<j} d_G(v_i, v_j)$ be the Wiener index of a connected graph $G$ of order $n$. Recall that $\mu(G) = \max_{x \in \mathbb{R}^n} \frac{x^T D(G) x}{x^T x}.$ 	Let $\mathbf{1} = (1, 1, \dots, 1)^T$. Then we obtain
	\[
	\mu(G) = \max_{\mathbf{x} \neq \mathbf{0}} \frac{\mathbf{x}^T D(G) \mathbf{x}}{\mathbf{x}^T \mathbf{x}} \geq \frac{\mathbf{1}^T {D}(G) \mathbf{1}}{\mathbf{1}^T \mathbf{1}} = \frac{2W(G)}{n}.
	\]
	Let $(d_1, d_2, \dots, d_n)$ be the degree sequence of $G$. By calculation, we obtain
	\begin{align*}
W(G)&= \frac{1}{2} \sum_{i=1}^n Tr_i(G)\\
	&\geq \frac{1}{2} \sum_{i=1}^n \left[ d_i + 2(n-1-d_i) \right]\\
	&= n(n-1) - \frac{1}{2} \sum_{i=1}^n d_i\\
	&= n(n-1) - e(G)\\
\end{align*}

	Hence, 	we have $\mu(G) \geq \frac{2W(G)}{n} \geq 2(n-1) - \frac{2e(G)}{n}$, which implies $$e(G)\ge n(n-1)-\frac{n}{2}\mu(G).$$
	Hence
	$$e(G)\ge n(n-1)-\frac{n}{2}\mu(G)\ge \binom{n - b - 1}{2} + ab + 2a.$$
By Theorem \ref{T-2.1}, $G$ contains an $[a, b]$-factor.
\end{proof}

\section*{Declarations}

The authors declare that they have no conflict of interest.

\section*{Data availability}

No data was used for the research described in the paper.

\end{document}